\documentclass[reqno, 11pt, a4paper, oneside]{amsart}
\usepackage{amsmath}
\usepackage{amsfonts}
\usepackage{amssymb}
\usepackage[ansinew]{inputenc}
\usepackage{graphicx}
\usepackage{amsbsy}
\usepackage{fancyhdr}
\usepackage{yfonts}
\usepackage{hyperref}
\usepackage{enumerate}

\numberwithin{equation}{section}
\newtheorem{teo}{Theorem}[section]
\newtheorem{prop}[teo]{Proposition}
\newtheorem{lema}[teo]{Lemma}

\theoremstyle{definition}
\newtheorem{defi}[teo]{Definition}

\newtheorem{ej}[teo]{Example}

\theoremstyle{remark}

\title{The algebraicity of ill-distributed sets}
\author{Miguel N. Walsh}
\address{Departamento de Matemática, Facultad de Ciencias Exactas y Naturales, Universidad de Buenos Aires, 1428 Buenos Aires, Argentina}
\email{mwalsh@dm.uba.ar}
\thanks{The author was partially supported by a CONICET doctoral fellowship.}

\begin{document}

\def\Fqn{\mathbb{F}_q^n}
\def\Fq{\mathbb{F}_q}
\def\Di{\mathbb{D}}
\def\O{\mathcal{O}}
\def\modp{(\text{mod }p)}
\def\Zp{\mathbb{Z}_p}

\begin{abstract}
We show that every set $S \subseteq [N]^d$ occupying $\ll p^{\kappa}$ residue classes for some real number $0 \le \kappa <d$ and every prime $p$, must essentially lie in the solution set of a polynomial equation of degree $\ll (\log N)^C$, for some constant $C$ depending only on $\kappa$ and $d$. This provides the first structural result for arbitrary $\kappa<d$ and $S$.
\end{abstract}

\maketitle

\section{Introduction}

We are concerned with the distribution of subsets of $\mathbb{Z}^d$ in residue classes. Precisely, letting $N$ be an integer parameter going to infinity and writing $[N]^d = \left\{ 1 , \ldots, N \right\}^d$, we are interested in understanding the nature of those subsets $S \subseteq [N]^d$ that manifest irregular behavior in their distribution in residue classes. While we know that a random subset of $[N]^d$ will occupy $\gg p^d$ residue classes for most primes $p$ of reasonable size, it has recently emerged the possibility that the only way this random behavior may be altered is if the set possess some strong type of algebraic structure (see \cite{CL,G,Kb,W}, and in particular the work of Helfgott and Venkatesh \cite{HV}, for some discussions on this topic). This problem certainly belongs to the realm of inverse theorems in arithmetic combinatorics, where abnormal arithmetic behavior of an object is traced back to a precise type of structure in it (see for example \cite{BGT,GTZ,H,TV}). A particularly tantalizing aspect of this phenomenon are its potential application to sieve theoretical questions. 

By a strong algebraic structure, it is generally meant that the behavior of these special sets should be governed by some small polynomial, in the sense that (depending on the context) the set should essentially be either the image or the solution set of a small polynomial. To make this idea of smallness precise, we introduce the following definition.

\begin{defi}[Complexity]
We say a nonzero polynomial has complexity at most $C$ if it has degree at most $C$ and its coefficients are bounded by $N^C$. We will say a set $S \subseteq [N]^d$ has complexity at least $C$ if, for every $C'<C$, there is no polynomial of complexity at most $C'$ vanishing on this set.
\end{defi}

Results so far in the literature have shown either that some additional hypothesis on the size, or 'regularity', of an ill-distributed set implies a strong form of algebraic structure \cite{HV,W} (i.e. control by a polynomial of uniformly bounded complexity) or that an ill-distributed set manifesting properties antagonistic to strong algebraic structure must be small (see \cite{G}, the main result of which was improved by Bourgain, and subsequently by Green and Harper \cite{GH}). The purpose of this note is to show that \emph{every} strongly ill-distributed set must possess some form of algebraic structure. Precisely, we shall prove the following.

\begin{teo}
\label{main2}
Let $S \subseteq [N]^d$ occupy $\ll p^{\kappa}$ residue classes for every prime $p$ and some real number $0 \le \kappa<d$. Then, for every $\varepsilon>0$, there exists some nonzero $P \in \mathbb{Z}[x_1,\ldots,x_d]$ of complexity $\ll_{\kappa,d,\varepsilon} (\log N)^{\frac{\kappa}{d-\kappa}}$ vanishing on at least $(1-\varepsilon) |S|$ points of $S$.
\end{teo}

This result is sharp up to the exponent, since for every choice of $0 \le \kappa < d$, there is some $c>0$ and a set $S \subseteq [N]^d$ satisfying the hypothesis of the result, but such that no polynomial of complexity at most $(\log N)^{c}$ vanishes on a positive proportion of this set (see Section \ref{examples}). Furthermore, if the presently known bounds towards the size of a maximal one dimensional ill-distributed sets were sharp (see \cite{Elsholtz}), then it is easy to see that the exponent $\frac{\kappa}{d-\kappa}$ would in fact be optimal for every $\kappa$ (see Example \ref{products}), although we do not expect this to be the case. In particular, when $d=1$, Theorem \ref{main2} recovers the best known bound on the size of strongly ill-distributed sets. 

It should be noticed that in contrast to previous results in the literature, we do not require $\kappa$ to be an integer, thus yielding the first strucutural result for arbitrary real values of this parameter.

We shall actually prove the following refinement of Theorem \ref{main2}, which shows that it suffices to know that the set $S$ is badly distributed in residue classes with respect to very small prime moduli.

\begin{teo}
\label{specific}
For every positive integer $d$, and real $0 \le \kappa < d$, there exists $\tau=\tau(d,\kappa)>0$ such that the following holds. Write $\mathcal{P}_I$ for the primes in the interval
\begin{equation}
\label{interval}
 I=\left[ \tau(\log N)^{\frac{d}{d-\kappa}}, 2 \tau(\log N)^{\frac{d}{d-\kappa}} \right].
 \end{equation}
Then, for every $S \subseteq [N]^d$ occupying $\ll p^{\kappa}$ residue classes mod $p$ for every prime $p \in \mathcal{P}_I$, and every $\varepsilon>0$, there exists some nonzero $P \in \mathbb{Z}[x_1,\ldots,x_d]$ of complexity $\ll_{\kappa,d,\varepsilon} (\log N)^{\frac{\kappa}{d-\kappa}}$ vanishing on at least $(1-\varepsilon) |S|$ points of $S$.
\end{teo}

Clearly, this result is optimal, as can be seen from taking an arbitrary subset of $[N]^d$ of size $(\log N)^{\frac{d\kappa}{d-\kappa}}$ and applying Siegel's lemma (Lemma \ref{siegel}). Of course, the strength of this result lies in the assertion that there is no construction that can beat this bound.

As in \cite{W}, the strategy of proof will consist in finding a small 'characteristic' subset $\mathcal{C} \subseteq S$, such that if a polynomial of low complexity vanishes on this set, then it must necessarily vanish at a positive proportion of the points of $S$, and then use the smallness of $\mathcal{C}$ to guarantee the existence of such a polynomial. However, contrary to \cite{W}, where the construction of the characteristic subset required a rather delicate induction and the study of the 'genericity' of the fibers of the set, we rely instead on a random sampling argument. The reason we can do this here is because, contrary to \cite{W}, the size of $\mathcal{C}$ may be allowed to go to infinity with $N$ due to the weaker kind of structure we are trying to attain, thus allowing much more flexibility in our choice. As remarked before, a particular advantage of this is that it allows us to deal with all ill-distributed sets in a uniform way.

The rest of this note is organized as follows. In Section \ref{examples} we present several examples of ill-distributed sets that help understand our results. The proof of Theorem \ref{specific}, after some preliminary discussion of notation, is carried out in Section \ref{pf}.

\medskip
\noindent \emph{Acknowledgments.} That something like Theorem \ref{main2} should hold for integer values of $\kappa$ was conjectured by Ben Green. I would like to thank him and Harald Helfgott for some comments and discussions around these topics.

\section{Some examples of ill-distributed sets}
\label{examples}

In this section we give some examples of ill-distributed sets, including some of relatively high algebraic complexity. In particular, we will show that for every $0 \le \kappa < d$ there exists $c>0$ and sets satisfying the hypothesis of Theorem \ref{main2} such that every polynomial vanishing on these sets must have complexity $\gg (\log N)^c$. We will also include an easy example to the effect that any improvement to the exponent $\frac{\kappa}{d-\kappa}$ in Theorem \ref{main2} must necessarily involve a better understanding of the size of ill-distributed one-dimensional sets.

\begin{ej}[One-dimensional sets]
\label{one} 
We begin with the one-dimensional case, where for every $0 < \epsilon < 1$ it is easy to construct subsets of $[N]$ of size $\gg (\log N)^{\epsilon}$ (and therefore complexity at least $\gg (\log N)^{\epsilon}$) occupying less than $p^{\epsilon}$ residue classes for every prime $p$. Indeed, by the Chinese Remainder Theorem we know that there exists subsets of $[N]$ of size $\gg N^{\epsilon}$ occupying less than $p^{\epsilon}$ residue classes for every prime $p<\log N$. Since any choice of $\lfloor (\log N)^{\epsilon} \rfloor$ elements of this set will necessarily occupy less than $p^{\epsilon}$ residue classes for every prime $p \ge \log N$, the claim follows. In general, we know from the larger sieve that such a set may have size at most $\ll (\log N)^{\frac{\epsilon}{1-\epsilon}}$ (see \cite{Elsholtz}). Notice that this result also follows from Theorem \ref{main2}. 
\end{ej}

\begin{ej}[Algebraic sets of bounded complexity]
By the Lang-Weil inequality, the solution set of a nonzero polynomial $P \in \mathbb{Z}[x_1,\ldots,x_d]$ of bounded complexity occupies $O(p^{d-1})$ residue classes for almost every prime $p$.
\end{ej} 

\begin{ej}[Product sets]
\label{products}
When $d \ge 2$, the most obvious example of a set occupying $p^{\kappa}$ residue classes for some $\kappa<d$, which is not dominated by a polynomial of bounded complexity, arises from considering product sets containing small ill-distributed sets. For example, when $d=2$ we may consider the product set $S=[N] \times X$, where $X \subseteq [N]$ is a small set occupying $\ll p^{\epsilon}$ residue classes for some $0<\epsilon<1$ and every prime $p$. From the first example we know that such a set $X$ may be constructed to have size $\gg (\log N)^{\epsilon}$, forcing any polynomial that vanishes on $S$ to have complexity at least $(\log N)^{\epsilon}$. However, as it was mentioned before, the best bound known for the size of such a set $X$ is $\ll (\log N)^{\frac{\epsilon}{1-\epsilon}}$. Suppose then that there exists some $X$ of the above form of size $|X| \sim (\log N)^{\frac{\epsilon}{1-\epsilon}}$ and consider the product set $X^d=X \times \ldots \times X \subseteq [N]^d$. Then this set occupies at most $p^{\kappa}$ residue classes for every prime $p$, with $\kappa=d \epsilon$, while every polynomial vanishing on this set must have complexity at least $|X| \gg (\log N)^{\frac{\epsilon}{1-\epsilon}}$. Since $\frac{\epsilon}{1-\epsilon}=\frac{\kappa}{d-\kappa}$, this would make Theorem \ref{main2} sharp for every value of $\kappa$. This shows that any improvement to the exponent in Theorem \ref{main2} must necessarily involve a better understanding of these types of sets.
\end{ej}

\begin{ej}['Perturbations' of algebraic sets]
We now show that products are not the only way to construct high-dimensional ill-distributed sets out of lower-dimensional ones. Indeed, let $Y \subseteq [N]$ be a set occupying $\ll p^{\varepsilon}$ residue classes for every prime $p$ and let $f \in \mathbb{Z}[x]$ have complexity at most $d$. Then we see that the set
$$ \left\{ (x,f(x)y) : x \in [N], y \in Y \right\},$$
occupies $\ll p^{1+\varepsilon}$ residue classes for every prime $p$, but every polynomial vanishing on this set has complexity at least $\gg_d |Y|$ (which we may take to be $\gg_d (\log N)^{\varepsilon}$ by Example \ref{one}).
\end{ej}

\begin{ej}
Since polynomials of rather large complexity may still possess a growing number of points, it is possible to construct pathological sets as before of larger algebraic complexity. We show an example of this. Fix a pair $0 < \delta, \rho < 1$ and let $d=(\log N)^{1-\delta}$. Consider the polynomial $f(x)=N^{\rho}x^d$, whose graph we know to have $\sim (N)^{\frac{1-\rho}{d}} = \exp((1-\rho)(\log N)^{\delta})$ integer points in $[N]^2$. Let us note that the only purpose of taking $\rho>0$ is to illustrate in a trivial manner how the complexity may also increase in the coefficients of the polynomials involved. Now let $Z$ be a set of size $\gg (\log N)^{\epsilon}$ occupying less than $p^{\epsilon}$ classes for every prime $p$ (constructed as in Example \ref{one}). Then it is clear that the set
$$ \left\{ (x,f(x)+z) : x \in [N], z \in Z \right\},$$
occupies less than $p^{\kappa}$ residue classes for every prime $p$, with $\kappa=1+\epsilon$. However, it is not hard to see (using the size of the graph of $f$ and the irreducibility of $f(x)+z-y$ as a polynomial in $x,y$, for every $z \in Z$) that if $c<\kappa-\delta$, then there is no polynomial of complexity at most $(\log N)^{c}$ vanishing on this set. Since $\delta>0$ could be chosen arbitrarily small, we see that the exponent can be taken arbitrarily close to $\kappa$.
\end{ej}

\section{The proof of Theorem 1.3}
\label{pf}

\subsection{Notation}

We will let $N$ be an integer parameter going to infinity. We will use the usual asymptotic notation $\ll, O$, where in particular the implicit constants are not allowed to depend on $N$. Any dependence of these implicit constants on the parameters $\kappa$ and $d$ present in the statement of Theorem \ref{specific} will be indicated by a subscript (e.g. $X \ll_{\kappa,d} Y$). Given a statement $Q(x)$ with respect to an element $x \in [N]^d$, we will denote by ${\bf 1}_{Q(x)}$ the function that equals $1$ when $Q(x)$ is true, and vanishes otherwise. We will also write $\Zp := \mathbb{Z}/p \mathbb{Z}$.

\subsection{Proof of Theorem 1.3}

The remaining part of this note is devoted to the proof of Theorem \ref{specific}. Precisely, we will show the existence of some absolute constant $\delta > 0$ such that, for every set $S$ as in the statement of Theorem \ref{main2}, there exists some nonzero $P \in \mathbb{Z}[x_1,\ldots,x_d]$ of complexity $\ll_{\kappa,d} (\log N)^{\frac{\kappa}{d-\kappa}}$ vanishing on at least $\delta |S|$ points of $S$. It is clear that Theorem \ref{main2} follows upon $O_{\varepsilon}(1)$ iterations of this result. From now on, let all the hypothesis and notations be as in the statement of that theorem.

As in \cite{W}, the main idea of the proof will be to construct a 'characteristic' subset $\mathcal{C}$ of $S$. By this we mean a set $\mathcal{C} \subseteq S$ such that if a polynomial of low complexity vanishes at $\mathcal{C}$, then it must also vanish at a positive proportion of $S$. Since we will also guarantee this set $\mathcal{C}$ to be small, we will be able to realize $\mathcal{C}$ as the solution set of a low complexity polynomial (by an application of Siegel's lemma), from where the result will follow.

The reason that allows us to find such a characteristic subset $\mathcal{C} \subseteq S$ is that, since $S$ occupies few residue classes for most primes $p$, it shall be possible to find a small set $\mathcal{C}$ such that for many elements $s \in S$, there are many primes $p$ for which $s \equiv c \, \modp$ for some $c \in \mathcal{C}$. But if we are then given a polynomial $P$ that vanishes on $\mathcal{C}$, and since polynomials descend to congruence classes, it follows that the above set of primes divides $P(s)$. If this polynomial $P$ is of bounded complexity, $|P(s)|$ must be small, and the only way this can be compatible with many primes dividing it is if $P(s)=0$.

From now on, let 
\begin{equation}
\label{r}
r = \eta (\log N)^{\frac{d\kappa}{d-\kappa}},
\end{equation}
be a parameter with the choice of the constant $\eta$ to be specified later. This parameter $r$ will be the size of the characteristic subset $\mathcal{C} \subseteq S$ that we shall construct below.

We will restrict considerations to the set of primes $\mathcal{P}_I$, which we recall, consists of those primes lying in the interval 
\begin{equation}
\label{tau}
I=[\tau(\log N)^{\frac{d}{d-\kappa}},2\tau(\log N)^{\frac{d}{d-\kappa}}],
\end{equation}
where $\tau>0$ is a constant that will be chosen later. 

Since we are interested in how well general elements of $S$ may be represented mod $p$ by fixed elements of $S^r$, we will be considering the set of pairs $(x,L) \in S \times S^r$. We say a tuple $(x,L) \in S \times S^r$ is {\sl good} mod $p$ if the residue class of $x$ mod $p$ is also the residue class of some element in $L$ and we say it is {\sl bad} otherwise. Our goal is to show that there exists some $\mathcal{C} \in S^r$ such that for a positive proportion of $x \in S$, the tuple $(x,\mathcal{C})$ is good mod $p$ for many primes $p \in \mathcal{P}_I$. This is summarized in the following proposition.

\begin{prop}
\label{tuple}
There exist sets $\mathcal{C},S' \subseteq S$ of size $|\mathcal{C}|=r$, $|S'| \gg |S|$, such that for every $x \in S'$ we have
\begin{equation}
\label{ema}
\sum_{p \in \mathcal{P}_I} {\bf 1}_{\exists c \in \mathcal{C} : x \equiv c \modp} \log p \gg |I|.
\end{equation}
\end{prop}

We begin by showing that, because of our choice of $r$, a positive proportion of the tuples of $S \times S^r$ are good.

\begin{lema}
\label{good}
There exists $c > 0$ such that, for every prime $p \in \mathcal{P}_I$, there are at least $c |S|^{r+1}$ good mod $p$ tuples in $S \times S^r$.
\end{lema}

\begin{proof}
For a residue class $a (\text{mod }p)$ in $\mathbb{Z}_p^d$ write $S_a$ for the probability of an element of $S$ lying in $a$ (that is, the proportion of elements of $S$ lying in this class). Then the probability of a tuple being bad equals
\begin{equation}
\label{prob}
 \sum_{a \in \mathbb{Z}_p^d} S_a \left( 1-S_a \right)^r.
 \end{equation}
It will suffice for our purposes that this sum be bounded away from $1$. Consider first those residue classes for which $S_a \le 1/r$. Then the sum (\ref{prob}) restricted to these classes is $O(p^{\kappa}/r)$ (since $S$ occupies $O( p^{\kappa})$ residue classes). Since $p \in \mathcal{P}_I$, we can make this arbitrarily small upon forcing the constants $\eta$ in (\ref{r}) and $\tau$ in (\ref{tau}) to satisfy
\begin{equation}
\label{cond1}
\eta \ge C_1 \tau^{\kappa},
\end{equation}
for some large value of $C_1$. The sum over the remaining classes is bounded by
$$ \left( 1 - \frac{1}{r} \right)^r \sum_{a \in \mathbb{Z}_p^d} S_a = \left( 1 - \frac{1}{r} \right)^r \rightarrow e^{-1},$$
as $r \rightarrow \infty$. Thus, letting $N$ be sufficiently large, we obtain the desired result. 
\end{proof}

\begin{proof}[Proof of Proposition \ref{tuple}]
Lemma \ref{good} clearly implies that for every $p \in \mathcal{P}_I$ there exist absolute constants $c_1$ and $c_2$ (independent of $p$) such that for at least $c_1|S|^r$ choices of $L \in S^r$, there are at least $c_2|S|$ elements $x \in S$ for which $(x,L)$ is a good mod $p$ tuple. Indeed, for this to fail we need
$$ c_1 |S|^{r+1}+(1-c_1)c_2 |S|^{r+1} \ge c |S|^{r+1},$$
and this clearly can be contradicted upon choosing the constants sufficiently small.

We say an element $L \in S^r$ is good mod $p$ if $(x,L)$ is good mod $p$ for at least $c_2|S|$ elements $x \in S$. By the above observations we know that
$$ \sum_{p \in \mathcal{P}_I} \left| \left\{ L \in S^r : L \text{ is good mod }p \right\} \right| \ge c_1 |S|^r |\mathcal{P}_I|.$$
It follows that there must exist some $\mathcal{C} \in S^r$ that is good mod $p$ for at least $c_1 |\mathcal{P}_I|$ primes in $\mathcal{P}_I$. From now on fix such a choice of $\mathcal{C}$.

By construction we know that
$$ \sum_{x \in S} \left| \left\{ p \in \mathcal{P}_I : (x,\mathcal{C}) \text{ is good mod }p \right\} \right| \ge c_1 c_2 |S||\mathcal{P}_I|.$$
But now it follows from this that there exist constants $c_3,c_4 > 0$ and a subset $S' \subseteq S$ of size $|S'| \ge c_3 |S|$, such that for every $x \in S'$ there are at least $c_4 |\mathcal{P}_I|$ primes $p \in \mathcal{P}_I$ for which $(x,\mathcal{C})$ is good mod $p$. In particular, we have that for every $x \in S'$ it must be
\begin{equation}
\label{cont}
\sum_{p \in \mathcal{P}_I} {\bf 1}_{\exists c \in \mathcal{C} : x \equiv c \modp} \log p \gg \frac{|I|}{\log |I|}\log |I| \gg |I|,
\end{equation}
and this completes the proof of Proposition \ref{tuple}
\end{proof}

In order to realize the set $\mathcal{C}$ provided by Proposition \ref{tuple} as the solution set of a polynomial equation of low complexity, we will need to employ a well known lemma of Siegel. We state it in the particular case needed and include a proof for completeness.

\begin{lema}[Siegel's lemma for polynomials]
\label{siegel}
Let $D > 0$ be an integer and let $\Sigma$ be some set of integer points in $[N]^d$ with $|\Sigma|<R:=\binom{D+d}{d}$. Then there exists a nonzero polynomial
$$ f= \sum_{i_1, \ldots, i_d} c_{i_1,\ldots,i_d} x^{i_1}x^{i_2}\ldots x^{i_d},$$
of degree at most $D$, vanishing at $\Sigma$, and where the coefficients are integers satisfying the bound
\begin{equation}
\label{bound}
 |c_{i_1,\ldots,i_d}| \le 4\left( R N^{D} \right)^{\frac{|\Sigma|}{R-|\Sigma|}}.
 \end{equation}
\end{lema}

\begin{proof}
Remember that we write $R={D+d \choose d}$ as in the statement. Let $H$ be a real parameter to be specified and write $\mathcal{R}$ for the set of polynomials in $\mathbb{Z}[x_1,\ldots,x_d]$ of degree at most $D$ and with coefficients of absolute value at most $H$. It is clear that $|\mathcal{R}|=(2H+1)^{R}$. For every element of $\mathcal{R}$, its values at $[N]^d$ are bounded in absolute value by $R H N^D$. It then follows that there are at most $(2 R H N^D+1)^{|\Sigma|}$ configurations of values that it can take over $\Sigma$. But then, if
$$ |\mathcal{R}|=(2H+1)^{R} > \left(2R H N^D+1 \right)^{|\Sigma|},$$
there must exist some pair of elements $f,g \in \mathcal{R}$ coinciding over $\Sigma$, from where it follows that $f-g$ vanishes identically on this set. Since the above inequality is satisfied with 
$$ H=2\left( R N^{D} \right)^{\frac{|\Sigma|}{R-|\Sigma|}},$$
the result follows.
\end{proof}

We can now complete the proof of Theorem \ref{specific}. Let $\mathcal{C} \in S^r$ be a tuple satisfying the conclusion of Proposition \ref{tuple}. Applying Lemma \ref{siegel} with $\Sigma = \mathcal{C}$ and $D$ chosen to satisfy
$$ \left( d! |\mathcal{C}| \right)^{1/d} = (d! r)^{1/d} < D \ll_d r^{1/d},$$
$$ \frac{r}{\binom{D+d}{d}-r} \ll 1,$$
we obtain a nonzero polynomial $P \in \mathbb{Z}[x_1,\ldots,x_d]$, of complexity $\ll_d r^{1/d}$, that vanishes on $\mathcal{C}$.

Let $S'$ be as in the statement of Proposition \ref{tuple}. We will show that $P$ must also vanish at every element of $S'$, thus finishing the proof of Theorem \ref{specific}. To see this, notice that we have $|P(x)| \le N^{c_d r^{1/d}}$ for every $x \in [N]^d$ and some constant $c_d$ depending only on $d$. Hence, we have in particular that for every such $x$, either
\begin{equation}
\label{probe}
 \sum_{p \in \mathcal{P}_I} {\bf 1}_{p|P(x)} \log p \ll_d r^{1/d} \log N \ll_d \eta^{1/d} (\log N)^{\frac{d}{d-\kappa}},
 \end{equation}
or it must be $P(x)=0$ otherwise. Suppose now $x \in S'$. Since $P$ vanishes on $\mathcal{C}$, we know that every prime $p$ for which $(x,\mathcal{C})$ is good will contribute to the left hand side of (\ref{probe}). Therefore, by Proposition \ref{tuple}, the left hand side of (\ref{probe}) is $\gg |I| \gg \tau (\log N)^{\frac{d}{d-\kappa}}$. Theorem \ref{specific} would then follow if we could force $\eta$ and $\tau$ to satisfy
$$ \tau \ge C_2 \eta^{1/d},$$
for some large $C_2$ depending on $d$. Recall that the only other condition needed between these constants is given by (\ref{cond1}). Since $\kappa<d$, it is clear that both inequalities can be satisfied if
$$ \eta \ge \left( C_1 C_2^{\kappa} \right)^{\frac{d}{d-\kappa}},$$
and so this concludes the proof of our result.

\end{document}